\documentclass[a4paper,12pt,pdftex,reqno]{amsart}

\usepackage[all]{xy}
\usepackage[latin1]{inputenc}
\usepackage{amscd}
\usepackage{rotating}
\usepackage[english]{babel}
\usepackage{amsmath,amssymb,amsfonts, amsthm,amscd, latexsym, marginnote}
\usepackage{tikz}
\usetikzlibrary{decorations.pathmorphing,backgrounds,fit,calc,through}
\usepackage{graphicx}
\usepackage{pgffor}
\usepackage{caption,subcaption}

\newcommand{\arr}[0]{{\mathcal A}}

\newcommand\al{\alpha}
\newcommand\tal{\tilde{\alpha}}

\newcommand{\A}[0]{\mathcal{A}}

\newcommand{\MA}[0]{\mathcal M(\A)}

\renewcommand{\L}[0]{\mathcal{L}}

\newcommand{\Z}[0]{\mathbb{Z}}
\newcommand{\R}[0]{\mathbb{R}}
\newcommand{\C}[0]{\mathbb{C}}

\newcommand{\St}[0]{\mathcal S}

\newcommand{\p}[0]{p}

\def\qed{\ifmmode $\Box$ \else{\unskip\nobreak\hfil
\penalty50\hskip1em\null\nobreak\hfil $\Box$
\parfillskip=0pt\finalhyphendemerits=0\endgraf}\fi}

\newcommand{\eq}[1][r]
       {\ar@<-3pt>@{->}[#1]
        \ar@<-1pt>@{}[#1]|<{}="gauche"
        \ar@<+0pt>@{}[#1]|-{}="milieu"
        \ar@<+1pt>@{}[#1]|>{}="droite"
        \ar@/^2pt/@{-}"gauche";"milieu"
        \ar@/_2pt/@{-}"milieu";"droite"}
\newcommand{\imm}[1][r] {\ar@{^{(}->}[#1]}

\newtheorem{df}{Definition}[section]

\newtheorem{teo}[df]{Theorem}

\newtheorem{lem}[df]{Lemma}
\newtheorem{cor}[df]{Corollary}
\newtheorem{rmk}[df]{Remark}
\newtheorem{rem}[df]{Remark}

\begin{document}

\hyphenation{mul-ti-pli-ci-ty ho-mo-lo-gy}

\title[Some computations on the characteristic variety of a line arrangement]{Some computations on the characteristic variety of a line arrangement}

\author[O. Papini and M. Salvetti]{ O. Papini}
\address{Department of Mathematics,
University of Pisa, Pisa Italy}
\email{papini@student.dm.unipi.it}

\author[]{ M. Salvetti}
\address{Department of Mathematics,
University of Pisa, Pisa Italy}
\email{salvetti@dm.unipi.it}
\thanks{Partially supported by: progetto PRA "Geometria e Topologia delle Variet\`a" (2018), Universita' di Pisa; INdAM; MIUR}



\maketitle

\begin{abstract} We find monodromy formulas for line arrangements which are fibered with respect to the  projection from one point. We use them to find $0$-dimensional translated components in the first characteristic variety of the arrangement $\mathcal R(2n)$ determined by a regular $n$-polygon and its diagonals.

\end{abstract}

\section{Introduction} An abelian local system $\L_{\rho}$ on the complement $\MA$ of a hyperplane arrangement $\A$ in $\C^N$ is defined by choosing  one non-zero complex number $\rho_{\ell}$ for each hyperplane $\ell\in \A.$ For a generic choice of the parameters $\rho_{\ell},$ it is well known that that homology concentrates in the top dimension $N$ (see for ex. \cite{kohno}).  The \emph{$i$-th characteristic variety} of $\A$ is the subvariety 
$$V_i(\A)=\{(\rho_\ell)_{\ell\in\A}\in (\C^*)^r:\ dim(H_i(\MA;\L_{\rho}))>0\}$$
where $r=\#\A.$ 
 A large literature on this subject is known, in connection also with the theory of\emph{ resonance varieties} and the study of the cohomology of the associated Milnor fiber
(see for example \cite{dimpapsuc}, \cite{dimpapa}, \cite{papamaci}, \cite{papadima1}, \cite{arapura}, \cite{suciu}, \cite{suciu2}, \cite{YoBand2}, \cite{YoBand3}, \cite{libgober},\cite{libyuz}, \cite{falk}, \cite{falkyuz}, \cite{denhamsuciu}, \cite{cohensuciu},\cite{cohenorlik}, \cite{gaiffi_salvetti}, \cite{SSv1},\cite{callegaro:classical}).

Probably the main problem is to understand if the characteristic varieties are combinatorially determined (see for example \cite{libgober1}).
This is known to be true for their "homogeneous part", which corresponds to  the resonance  variety by the tangent cone theorem (see for ex. \cite{dimcabook} for general references). Nevertheless, the translated components of the characteristic variety are still not well understood; in particular, the geometric description which is known for the translated components of dimension at least one does not work in the same way for the $0$-dimensional translated components (see also \cite{artal}). 

For these reasons we think that it can be useful to produce more examples (interesting in themselves) such that the characteristic variety has some  translated $0$-dimensional global component. 
\smallskip

In this paper we outline a different approach, based on an (apparently new) elementary description of the characteristic variety.  We consider here  the case $N=2.$  We find  that the arrangement $\mathcal R(2n)$ determined by a regular $n$-polygon and its diagonals produces the above mentioned  phenomenon for $n\geq 5,$ namely we find $\phi(n)$ translated $0$-dimensional components in its characteristic variety (here $\phi$ is the Euler function). This fact was experimentally observed in \cite{papini} by using computer methods, up to $n=7.$  
 
The arrangement $\mathcal R(2n)$ belongs to the class of \emph{fibered} arrangements: the projection through its ''center'' gives a fibration of the complement $\MA$ over $B=\C \setminus \{n-1\ points\}$ with fiber $F=\C\setminus \{n\ points\}.$     

We obtain the above result by using algebraic complexes which compute the parallel transport and the monodromy of the first homology group of the fiber, which we think are interesting in themselves.   For example, we deduce  a restriction on global components of the characteristic variety (thm \ref{restriction}) and a description of $V(\A)\ (=V_1(\A))$ as the set of points such that the transpose of the monodromy operators have a common eigenvector  (thm \ref{charmono}).

For the arrangement $\mathcal R(2n),$ we use the algebraic complexes constructed in section \ref{sec2} to make explicit computations. If $\omega_n$ is an $n$-th primitive root of $1,$ we  find that 
assigning to the edges of the polygon the parameter $(\omega_n)^k$ and to the diagonals the parameter $(\omega_n)^{k(n-2)}$, $k=1,\dots,n-1,$ lowers the rank of the boundary operator in the algebraic complex computing $H_0(B;H_1(F;\L_{|F})),$ therefore the corresponding point $P_{n,k}$ belongs to $V(\A).$ In case $k$ and $n$ are coprime, we deduce that $P_{n,k}$ is an isolated point in the characteristic variety.  Our argument works for $n\geq 5$ and uses the description (see thm \ref{charmono}) of the characteristic variety and the remark that triple points give  simple eigenvalues for the corresponding monodromy operators. We find  that the shape of the common eigenvector imposes equalities among the eigenvalues of the monodromy operators, which define a zero-dimensional locus.

 When $n$ and $k$ are not coprime, our argument can also be used. In this case the shape of the common eigenvector imposes less equalities (so, a higher dimensional locus in general). For example, we find in case \  $n=4m$ \ an explicit  $1$-dimensional translated component containing $P_{n,m}$ and $P_{n,3m}$ (thm \ref{onedimensional}). When  $n=4,$ this gives another description of the $1$-dimensional translated component found in \cite{suciu}. For $n=8,$ the translated $1$-dimensional component contains the two points $P_{8,2},\ P_{8,6}$  (in contrast with \cite{toryos}, p, 45--46;
 see remark \ref{nonprimo}).
 
We also remark that our approach can be generalized (with some modifications) to not necessarily fibered arrangements (we will return to this in future work).

\section{Algebraic complexes}\label{sec2}   In this section we produce an algebraic complex computing local homology for a line arrangement. This complex seems particularly convenient in case of a fibered arrangement, in which case it is smaller than the one in \cite{gaiffi_salvetti}, \cite{SS}.

Let $\A$ be an affine line arrangement; we assume that $\A$ is defined over $\R$ because formulas are easier to write, but all we say can be generalized to general complex arrangements.  Let $\MA$ be its complement in $\C^2=\{(x,y):\ x,y\in\C\};$ let $\pi:\C^2\to\C$ be the projection $\pi(x,y)=x$ onto the $x$-axis and $\pi'=\pi_{|\MA}:\MA\to \C$ its restriction.  Let $\St\subset \A$ be the set of singular points of the arrangement, $\pi(\St)\subset \C$ its projection and $B=\C\setminus \pi(\St).$ Then 
\begin{equation}\label{bundle} \pi''=\pi'_{| (\pi')^{-1}(B)}:\  (\pi')^{-1}(B) \to B\end{equation}
is a fiber bundle with fibers $F_x:=(\pi')^{-1}(x),\ x\in B.$ We divide the arrangement 
$$\A=\{\ell_1,\dots\ell_n,\ell'_1,\dots,\ell'_m\}$$
into  \emph{horizontal} and  \emph{vertical} lines,  meaning that $\ell'_j=\pi^{-1}(\xi)$ for some $\xi\in \pi(S),$ while $\pi(\ell_i)=\C.$ Notice that if $\forall \xi\in \pi(S)$ the vertical line $\pi^{-1}(\xi)$ is in $\A,$ then $\MA=(\pi')^{-1}(B)$ so all the complement $\MA$ is a fiber bundle over $B.$  Let $\C_x=\pi^{-1}(x),\ x\in B;$ then $F_x= \C_x\setminus \{y_1,\dots,y_n\}$ where  $y_i= \C_x\cap \ell_i ,$ $i=1,\dots,n.$        

We fix a fiber $F_0=(\pi'')^{-1}(x_0),\ x_0\in B,$ $x_0\in\R;$ take a basepoint $P_0\equiv(x_0,y_0)\in F_0$ with $\Im (y_0)>\!\!>0,$ and elementary well-ordered generators $\al_1,\dots\al_n$ of $\pi_1(F_0),$ where $\al_i$ is constructed by using a path connecting $P_0$ to a small circle around $\ell_i\cap F_0.$ The indices of the lines are taken according to the growing intersections with the real axis $\Re(\C_{x_0}).$ 

We can make the same construction for any fiber $F_x,\ x\in B, \ x\in\R,$ 
by using the basepoint $P_0^{(x)}\equiv(x,y_0)$ and generators $\al^{(x)}_i;$ the indices here are locally computed according to the growing intersections of the lines with $\Re(\C_{x}).$  

We are interested in the local homology of $\MA,$ where we consider the abelian local system defined by taking $\C$ as an  $H_1(\MA;\Z)$-module: the action is given by taking standard generators  (small circles around the hyperplanes) into multiplication by non-zero numbers. So, such local systems correspond to $(n+m)$-tuples of parameters in $(\C^*)^{n+m}.$  

We call $s_1,\dots, s_n$ the parameters which correspond to the lines $\ell_1,\dots,\ell_n$ and $t_1,\dots,t_m$ those corresponding to the lines $\ell'_1,\dots,\ell'_m$ respectively. Denote by $\L=\L_{s,t}$ the corresponding local system. We also set $\L_s=\L_{|_{F_0}}.$  

In this paper we will assume that $\MA$ fibers over $B;$  that corresponds, as said before, to the case where the singular set $S$ is contained into the union of the vertical lines. 
\bigskip

\begin{lem} We have
\begin{equation}\label{decompo} H_1(\MA;\L_{s,t}) \ =\ H_0(B;H_1(F_0;\L_s))\oplus H_1(B;H_0(F_0;\L_s)) \end{equation}
where the second factor in the right member vanishes if $\L_s$ is non trivial (i.e., some $s_i\not= 1$).
\end{lem}
\begin{proof} It immediately derives from the associated spectral sequence and from the definition of the $0$-th local homology group. \end{proof}
\bigskip

So we have to compute the term $H_0(B;H_1(F_0;\L_s)),$ where $H_1(F_0;\L_s)$ is seen as a $\pi_1(B)$-module (actually, an $H_1(B;\Z)$-module), the action being given by the monodromy of the bundle. 
\bigskip

\begin{lem} One has   $dim(H_1(F_0;\L_s))= n-1$  if $\L_s$ is non-trivial. As generators we can take the classes of 
\begin{equation}\label{generat0} \tal_{i,i+1}\ =\ (1-s_{i+1})\al_i-(1-s_i)\al_{i+1}, \quad i=1,\dots, n-1. \end{equation}
\end{lem}
\begin{proof} The fiber $F_0$ deformation retracts onto a wedge of $n$ $1$-spheres, given by $\cup_{i=1}^n\ \al_i.$ Therefore by standard methods to compute the homology we are done.\end{proof}
\bigskip

To compute the monodromy on $H_1,$ we need to understand the parallel transport 

\begin{equation}\label{transp} \gamma(x,x')_*:H_1(F_x;\L_x)\to H_1(F_{x'};\L_{x'})\end{equation}
\bigskip

\noindent where $\gamma(x,x')$ is a path in $B$ connecting $x$ with $x'$ and $\L_x=\L_{| F_x}.$ We compute the parallel transport for any points $x, x' \in \R\setminus\pi(S),$ where $\R$ is the real axis of the first coordinate (recall that we are assuming $\pi(S)\subset \R$).
We denote by $\tal^{(x)}_{i,i+1}, \ i=1,\dots,n-1,$ the generators of $H_1(F_x;\L_x),$ constructed as those in (\ref{generat0}), by using the $\al^{(x)}_i$'s. 
Formulas look better if we take as generators 

\begin{equation}\label{generat} \al_{i,i+1}= \frac{\tal_{i,i+1}}{(1-s_i)(1-s_{i+1})}= {\frac{\al_i} {1-s_i}} - {\frac{\al_{i+1}}{1-s_{i+1}}}.\end{equation}
\bigskip

\noindent Of course, this requires each $s_i\not=1,$ but since we are interested in \emph{global components} of the characteristic variety (i.e., not contained in any coordinate tori $s_i=1$), this is not a serious restriction (general formulas for the $\tal_{i,i+1}$ are similar).

Given $x, x'\in\R\setminus\pi(S),$  we consider a path $\gamma(t)=\gamma(x,x')(t),\ t\in[0,1],$ in $B$ connecting $x, x',$ such that $\gamma \cap \R =\{x,x'\}$ and $\Im(\gamma(t))(x-x')\geq 0,\ t\in[0,1]$  (i.e., $\gamma$ leaves the real axis on the left while traveling from $x$ to $x'$). We denote by $\tau(x,x')$ the corresponding transport isomorphism (\ref{transp}). 

Let $\sigma=\sigma(x,x')$ be the permutation of the indices $1,\dots,n$ which is obtained as follows: the $i$-th line in the $x$-ordering (considering the growing intersections with $\Re(\C_x)$) is the $\sigma(i)$-th line in the $x'$-ordering.  For each $i=1,\dots,n-1$ we set
$\sigma(i,i+1)=+1$ or $-1$ depending on whether $\sigma(i+1)>\sigma(i)$ or $\sigma(i+1)< \sigma(i).$  
\bigskip

\begin{teo}{ \emph{[parallel transport]} We have 
\begin{equation}\label{transport} \tau(x,x')(\al^{(x)}_{i,i+1})\ =\ 
\sigma(i,i+1) \sum_{j=m(i,i+1)}^{M(i,i+1)-1}\ (\prod_{ \substack {k>i \\ \sigma(k)\leq j}} \ s_k)\ \al^{(x')}_{j,j+1}.\end{equation}
Here $m(i,i+1)=min(\sigma(i),\sigma(i+1)),$ and $M(i,i+1)=max(\sigma(i),\sigma(i+1)).$}
\end{teo}
\noindent\emph{Proof.} 
The proof is a standard computation and it is obtained by the following steps.

First, remark that any generator in (\ref{generat0}) is obtained from Fox calculus from the commutator 
$[\al_i,\al_{i+1}]$     as
\begin{equation}\label{foxcalculus}\tilde{\al}_{i,i+1} = \sum_{j=1}^n \varphi(\frac{\partial}{\partial \al_j} [\al_i,\al_{i+1}])\  \al_j\end{equation}
where $\varphi:\Z[\pi_1(F_0)]\to \Z[s_1^{\pm 1},\dots,s_n^{\pm1}]$ is the valuation homomorphism taking $\al_i$ to $s_i.$ 

Next, we remark that  while the path $\gamma$ turns around the projection of some singularity of a half-circle, the corresponding points in the vertical line make a half-twist, giving rise to a "local permutation" which takes a sequence of consecutive numbers to the opposite sequence (i.e.: $k,\dots,k+h$ goes to $k+h,\dots,k$). 

By induction on the number of points in $\pi(S)$ separating $x$ from $x'$ we easily see that  a generator  $\al^{(x)}_i$ is taken by $\gamma$ to $P\al^{(x')}_{\sigma(i)}P^{-1},$
where $P$ is the product (in increasing order) of the $\al^{(x')}_j$ such that $j<\sigma(i)$ and $\sigma^{-1}(j)>i.$ Then we apply Fox calculus as in (\ref{foxcalculus}) to the transform $\gamma_*([\al^{(x)}_i,\al^{(x)}_{i+1}])$ and we conclude.       \qed
\bigskip

The local monodromy around one point $p\in\pi(S)$ is obtained by taking a point $x\in\R$ very close to $p$ and transporting  the fiber of $\pi''$ starting from $x$  around a circle centered in $p$. If $x'\in\R$ is a point symmetric of $x$ with respect to $p,$ the local monodromy \ $\mu_x$ is the automorphism of $H_1(F_x;\L_x)$ which is  obtained by composing $\tau(x',x)\circ\tau(x,x').$ The "local permutation" $\sigma_{x,x'}$ is determined by a partition $\Delta_1,\Delta_2,\dots,\Delta_h$ of the set $\{1,\dots,n\},$ where each $\Delta$ is composed of consecutive numbers $a,a+1,\dots,a+r$ which are the indices (in the $x$-ordering) of the lines which intersect into a singular point $P$ in the vertical line $\C_p$ (so the multiplicity of $P$ as a singular point of $\A$ equals $|\Delta|+1$). 
 \bigskip

\begin{teo}\label{teo:local}{\emph{ [local monodromy]} }With the previous notations, the local monodromy around $p$ is given by

\begin{equation}\label{localmonodromy} \mu_x(\al^{(x)}_{i,i+1})\ =\ \end{equation} 

\begin{enumerate}
\item[]  $$=\ (\prod_{j\in\Delta} s_j) \  \al^{(x)}_{i,i+1} \hskip1.5cm  
  \mbox{if $i, i+1$ belong to the same block $\Delta$ of $\sigma_{x,x'}$;} $$ 
  
\item[] $$=\ \al^{(x)}_{i,i+1}\ +\ \sum_{k=a}^{i-1}(1-s_a\dots s_k)\al^{(x)}_{k,k+1} \ +\  
\sum_{k=i+2}^{b}s_{i+1}\dots s_{k-1}(1-s_k\dots s_b)\al^{(x)}_{k-1,k}  $$
 if $i$ is the last element of the block $\{a,a+1,\dots,i\}$ and $i+1$ is the first element 
 of the next block $\{i+1,i+2,\dots,b\}.$  
\end{enumerate}

%

\end{teo}
\begin{proof} The proof directly follows  from (\ref{transport}) and from $\mu_x(\al^{(x)}_{i,i+1})\ = \tau(x',x)\circ\tau(x,x').$ \end{proof}
\bigskip

As an immediate consequence of (\ref{localmonodromy}) we have 
\bigskip

\begin{cor} \label{cor:charpol}Let $p\in\pi(S)$ and  let $P_1,\dots,P_h$ be the singular points of $\A$ in the vertical line 
$\C_p,$ having multiplicity $m_1,\dots,m_h$ respectively ($m_1+\dots +m_h=n+h$). Then the local monodromy $\mu_x$ has characteristic polynomial
\begin{equation}\label{charpol} 
p_x(\lambda)= (\lambda-1)^{h-1} \prod_{i=1}^h\ (\lambda-\prod_{P_i\in\ell} s_{\ell})^{m_i-2}\ \end{equation}
and it is diagonalizable if all eigenvalues $\prod_{P_i\in\ell} s_{\ell}$ are different from $1.$
Here we use $s_{\ell}$ to indicate the parameter corresponding to the horizontal line $\ell.$
\end{cor}\qed
\bigskip

Now we consider the global monodromy $\mu:\pi_1(B,x_0)\to Aut(H_1(F_0;\L_s)).$  For each point $p\in\pi(S)$ we denote by $x_p, x'_p\in\R$ two points close to $p$ and lying in opposite sides with respect to $p$ in $\R;$ assume that $x_p$ is that of the two points lying in the segment  $[p,x_0]\subset\R.$ 

The easiest way to compute monodromy is to take generators $\delta_p=\gamma(x_0,x'_p)\gamma(x'_p,x_0)$ for $\pi_1(B,x_0);$ such generator circles around all the projections of the singular points $p'\in[p,x_0]$ counterclockwise. By using the parallel transport (\ref{transport}) we get
\bigskip

\begin{teo}{\emph{[global monodromy (1)]} } The global monodromy $\mu$ is determined by
$$ \mu([\delta_p]) (\al_{i,i+1})=\ \sum_{j=1}^{n-1} \mu_j^i\ \al_{j,j+1} $$ 
where

\begin{equation}\label{globalmonodromy1} 
\mu^i_j=\ \sigma(i,i+1)\sum_{\substack{[k,k+1]\subset[\sigma(i),\sigma(i+1)] \\ [j,j+i]\subset [\sigma^{-1}(k),\sigma^{-1}(k+1)]}}\ \sigma^{-1}(k,k+1)\ \prod_{\substack{h\geq i+1 \\ \sigma(h) \leq k}} s_h \prod_{\substack{l\geq k+1\\ \sigma^{-1}(l)\leq j}} s_{\sigma^{-1}(l)}\end{equation}
\bigskip

\noindent  where $[a,b]$ is meant to be the segment connecting the two points $a$ and $b$ in the real axis (of $\C_{x'_p}$ and $\C_{x_0}$). 
\end{teo}
\bigskip

It is convenient to take elementary generators of $\pi_1(B,x_0),$ associated to  $p\!\!:$\ the path $\gamma_p$  is composed in sequence  by 
$$\gamma_p=\gamma(x_0,x_p)\gamma(x_p,x'_p)\gamma(x'_p,x_p)\gamma(x_p,x_0)^{-1}.$$
Clearly such paths give a well-ordered set of elementary generators of $\pi_1(B,x_0),$ so the global monodromy is determined by their images. 
\bigskip
  
\begin{teo}{\emph{[global monodromy (2)]}} The global monodromy $\mu$ is determined by the maps 
\begin{equation}\label{globalmonodromy}\mu([\gamma_p]) = \tau(x_0,x_p)^{-1}\mu_{x_p}\tau(x_0,x_p)\end{equation}
where $\tau(x_0,x_p)$ is the parallel transport in  (\ref{transport}) and $\mu_{x_p}$ is the local monodromy in (\ref{localmonodromy}).
\end{teo}\qed
\bigskip

\noindent It is also possible to write explicit formulas for (\ref{globalmonodromy}): formulas (\ref{transport}) and (\ref{localmonodromy}) make possible to compute the images of the  generators $\al_{i,i+1};$ we do not report such formulas here because we don't need their explicit form now. Notice also that (\ref{globalmonodromy}) and (\ref{globalmonodromy1}) are related by 
$$\mu([\gamma_{p_i}])=\mu([\delta_{p_{i-1}}])^{-1}\mu([\delta_{p_{i}}])$$
where the $p_i$ are ordered according to increasing distance from $x_0.$
\bigskip

We now come back to the computation of  $H_0(B;H_1(F_0;\L_s))$ in (\ref{decompo}). 
The space $B$ deformation retracts onto a wedge of $1$-spheres, actually onto the union $\cup_{p\in\pi(S)}\ \gamma_p,$ where the unique $0$-cell is $x_0.$ 
By standard methods which compute local system homology we obtain
\bigskip

\begin{teo} \label{fundam} The vector space $H_0(B;H_1(F_0;\L_s))$ is the cokernel of the map 
\begin{equation}\label{conucleo} \partial_1:\oplus_{p\in\pi(S)} H_1(F_0;\L_s) [\gamma_p] \to  H_1(F_0;\L_s) [x_0]   \end{equation}
taking 
\begin{equation}\label{delta1}\partial_1([\gamma_p]) = t_p\mu([\gamma_p])-Id\end{equation}
\medskip

\noindent where we set here $t_p$ as the parameter associated to the vertical line $\C_p\in \A.$ 
\end{teo}\qed
\bigskip

\noindent It follows that $\partial_1([\gamma_p])$ has rank lower than $n-1$ iff $t_p$ coincides with the inverse of an eigenvalue of $\mu([\gamma_p]).$ Therefore either $t_p=1$ or \ $t_p\prod_{P_i\in\ell} s_{\ell}=1,$ \ for some singular point $P_i\in\C_p$ of order at least $3$ (see (\ref{charpol})). 

We obtain the following corollary. \bigskip

\begin{teo}\label{restriction} The global components of the characteristic variety of $\A$ are contained in 
\begin{equation}\label{mainideal} W(\A)=\bigcap_{p\in\pi(S)}\prod_{\substack{P\in \C_p \\ m(P)\geq 3}} \{ \prod_{\ell:P\in \ell} \rho_{\ell} - 1 \}\ \subset (\C^*)^{n+m}
\end{equation}
where we write here $\rho_{\ell}$ as the parameter associated to the line $\ell$ (in our previous notations it is of the shape $s_i$ for horizontal lines and $t_j$ for vertical lines).
\end{teo}\qed
\bigskip

As another  consequence we get the following description of the characteristic variety of $\A.$
\bigskip

\begin{teo}\label{charmono}
The characteristic variety $V(\A)$ coincides with the set of values $(s,t)\in(\C^*)^{n+m}$ such
that the transpose of the monodromy operators $\mu([\gamma_p])(s),\ p\in\pi(S),$ have a common eigenvector, relative to eigenvalues $t_p^{-1},\ p\in\pi(S).$ 
\end{teo}     \qed

\section{Some computations}
\noindent We consider here the arrangement $\mathcal R(2n)$ which is projectively given by taking the lines spanned by the edges of a regular $n$-polygon, together with all its $n$ diagonals (see figures \ref{picture1},  \ref{picture3} for the case $n=6,7$ respectively). 
Notice that there is an $n$-tuple point $P$ which is the intersection of the diagonals. For odd $n,$ on each diagonal we have, besides $P,$ one double point at the intersection with the middle point of an  edge, and $\frac{n-1}{2}$ triple points. For even $n,$ $n=2k,$ we have $k$ diagonals passing through two opposite vertices of the $n$-gon, each of them containing $k$ triple ponts; other $k$ diagonals passing through the middle points of two opposite sides of the $n$-gon, each of them containing two double points and $k-1$ triple points. Of course, the projection from $P$ makes the complement $\MA$ a fiber bundle. We take one of the diagonals to infinity, so the projection $\pi:\C^2\to\C$ describes $\MA$ as a fiber bundle over $\C\setminus\{n-1\ points\}$ with fiber $\C\setminus\{n\ points\}$ (see figures \ref{picture2}, \ref{picture4}).

Let $\ell'_1=\{x=p_1\},\dots,\ell'_{n-1}=\{x=p_{n-1}\}$ \ ($p_i\in\R$)\   be the vertical lines, ordered according to $p_1>\dots>\p_{n-1}.$  In the notations of the previous part, let $x_0\in\R$ a basepoint in the $x$-coordinate: we choose $x_0>\!\!>p_1.$  Let also $\ell_1,\dots,\ell_n$ be the horizontal lines, ordered according to increasing intersections with the real axis of $\C_0.$  Fix also points $x_i=p_{i}+\epsilon, \ x'_i= p_i-\epsilon, \ i=1,\dots, n-1,$ \ $0<\epsilon<\!\!<1.$    

We trivially have (see fig. \ref{picture2},
 fig. \ref{picture4}):
\bigskip

\begin{minipage}[t]{\linewidth \fboxsep \fboxrule}
\hspace*{3em}\raisebox{35em}

\centering
{\begin{tikzpicture}[scale=0.95]
\node[label={[xshift=0cm,yshift=0cm]11: $$}] (w1) at (1.02,1.77) {$\bullet$};
\node[label={[xshift=0cm,yshift=0cm]12: $$}] (w2) at (-1.04,1.76) {$\bullet$};
\node[label={[xshift=0cm,yshift=0cm]12: $$}] (w3) at (-2.06,-0.03) {$\bullet$};
\node[label={[xshift=0cm,yshift=0cm]12: $$}] (w4) at (-1.02,-1.81) {$\bullet$};
\node[label={[xshift=0cm,yshift=0cm]12: $$}] (w5) at (1.04,-1.8) {$\bullet$};
\node[label={[xshift=0cm,yshift=0cm]12: $$}] (w6) at (2.06,-0.01) {$\bullet$};

\node[label={[xshift=0cm,yshift=0cm]13: $$}] (z3) at (3.1,-1.79) {$$};
\node[label={[xshift=0cm,yshift=0cm]14: $$}] (z4) at (3.08,1.78) {$$};
\node[label={[xshift=0cm,yshift=0cm]15: $$}] (z5) at (-0.02,3.55) {$$};
\node[label={[xshift=0cm,yshift=0cm]16: $$}] (z6) at (-3.1,1.75) {$$};
\node[label={[xshift=0cm,yshift=0cm]17: $$}] (z7) at (-3.08,-1.82) {$$};
\node[label={[xshift=0cm,yshift=0cm]18: $$}] (z8) at (0.02,-3.59) {$$};

\foreach \source/\target in {z5/z3, z6/z4, z7/z5, z8/z6,z3/z7,z8/z4}
\draw[line width=1,shorten <=-2cm, shorten >=-2cm] (\source)--(\target);

\foreach \source/\target in {z6/z3, z7/z4, z8/z5}
\draw[line width=1,shorten <=-2cm, shorten >=-2cm] (\source)--(\target);

\foreach \source/\target in {w4/w1, w5/w2, w6/w3}
\draw[line width=1,shorten <=-3cm, shorten >=-3cm] (\source)--(\target);

\end{tikzpicture}}
\vskip2.5cm

 \centering
 
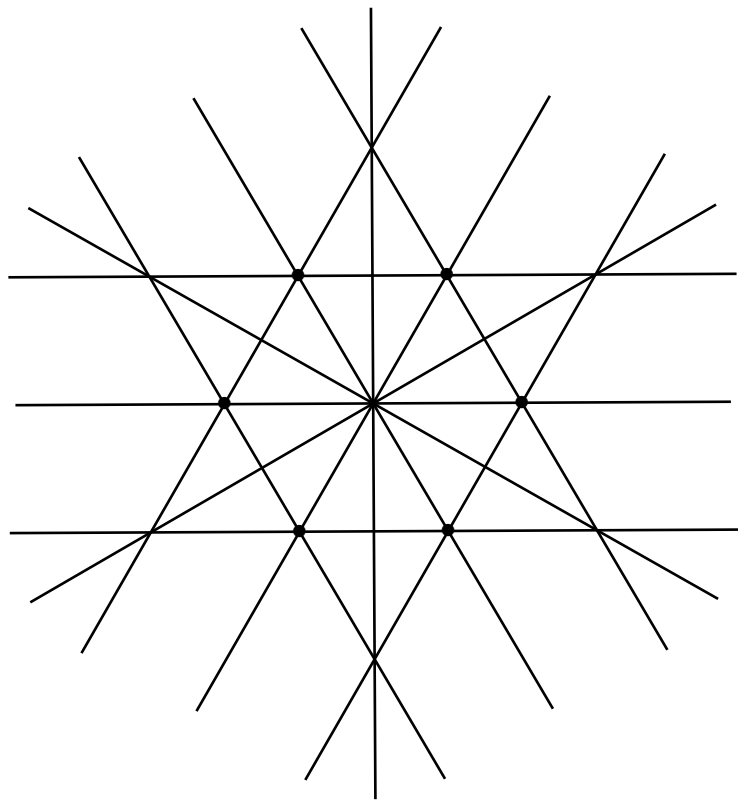
\captionof{figure}{The arrangement $\mathcal R(12)$ (projective picture)}
\label{picture1}
\end{minipage}

\bigskip

\begin{minipage}[t]{\linewidth\fboxsep\fboxrule}
\hspace*{3em}\raisebox{40em}

\centering
{\begin{tikzpicture}[scale=1]

\coordinate[label=150:$$] (a) at (8,-3.99);
\coordinate[label=100:$$] (b) at (8,-2);
\coordinate[label=50:$$] (c) at (8,2);
\coordinate[label=40:$$] (g) at (8,4);
\coordinate[label=20:$$] (d) at (5.67,0);
\coordinate[label=20:$$] (f) at (5.67,-2);
\coordinate[label=20:$$] (e) at (5.67,2);
\coordinate[label=20:$$] (i) at (4.51,1);
\coordinate[label=20:$$] (j) at (4.51,-1);
\coordinate[label=20:$$] (z) at (3.34,2);
\coordinate[label=20:$$] (r) at (3.34,-2);
\coordinate[label=20:$$] (p) at (3.34,-0.01);
\coordinate[label=20:$$] (w) at (0.99,2);
\coordinate[label=20:$$] (t) at (1.01,-2);
\coordinate[label=20:$$] (u) at (1.02,-3.99);
\coordinate[label=20:$$] (q1) at (8,-6.48);
\coordinate[label=20:$$] (q2) at (5.67,-6.48);
\coordinate[label=20:$$] (q3) at (4.51,-6.48);
\coordinate[label=20:$$] (q4) at (3.34,-6.48);
\coordinate[label=20:$$] (q5) at (1.04,-6.48);

\node[label={[xshift=0.1cm,yshift=-0.2cm]100:$p_1$}] (x1) at (8,-6.48) {$\bullet$};

\node[label={[xshift=-0.9cm,yshift=-0.18cm]35: $p_2$}] (x2) at (5.67,-6.48) {$\bullet$};
 \node[label={[xshift=0.2cm,yshift=0.8cm]600: $p_3$}] (x4) at (4.51,-6.48){$\bullet$};
\node[label={[xshift=-0.9cm,yshift=0.65cm]700:$p_4$}] (x5) at (3.34,-6.48){$\bullet$};
\node[label={[xshift=-0.7cm,yshift=0cm]800:$p_5$}] (x6) at (1.04,-6.48){$\bullet$};

\draw[line width=2,shorten <=-1.5cm, shorten >=-2cm] (q1)--(g) node[ pos=-0.17]{$\ell'_1$};
\draw[line width=2,shorten <=-2cm, shorten >=-1.5cm] (c)--(w)node[right, pos=-0.28]{$\ell_4$};
\draw[line width=2,shorten <=-2cm, shorten >=-7cm] (b)--(f) node[right, pos=-0.82]{$\ell_3$};
\draw[line width=2,shorten <=-1.5cm, shorten >=-3.8cm] (q2)--(e)node[pos=-0.21]{$\ell'_2$};
\draw[line width=2,shorten <=-3cm, shorten >=-5cm] (c)--(d)--(r) node[right,pos=-1.95]{$\ell_5$};
\draw[line width=2,shorten <=-3cm, shorten >=-4.5cm] (b)--(d)--(i)--(z)node[right,pos=-4.85]{$\ell_2$};
\draw[line width=2,shorten <=-3cm, shorten >=-3cm] (a)--(w)node[right,pos=-0.32]{$\ell_1$};
\draw[line width=2,shorten <=-3cm,, shorten >=-3cm] (g)--(t)node[right,pos=-0.32]{$\ell_6$};
\draw[line width=2,shorten <=-1.5cm,, shorten >=-10.1cm](q3)--+($(d)-(f)$) node[pos=-0.9]{$\ell'_3$};
\draw[line width=2,shorten <=-1.5cm,, shorten >=-10.1cm](q4)--+($(d)-(f)$)node[pos=-0.9]{$\ell'_4$};
\draw[line width=2,shorten <=-1.5cm,, shorten >=-10.1cm](q5)--+($(d)-(f)$)node[pos=-0.9]{$\ell'_5$};
\draw[line width=1,shorten <=-2.9cm,, shorten >=-7.2cm](q1)--+($(z)-(e)$)node[pos=-1.3]{$x$};

 \end{tikzpicture}}
 \captionof{figure}{The arrangement $\mathcal R(12)$ (affine picture) projected onto the 
$x$-axis}
\label{picture2}
\end{minipage}

\bigskip

\begin{minipage}[t]{\linewidth \fboxsep \fboxrule}
\hspace*{3em}\raisebox{35em}

\centering
{\begin{tikzpicture}[scale=0.55]
\node[label={[xshift=0cm,yshift=0cm]13: $$}] (z3) at (6.3,7.89) {$$};
\node[label={[xshift=0cm,yshift=0cm]14: $$}] (z4) at (8.1,0) {$$};
\node[label={[xshift=0cm,yshift=0cm]15: $$}] (z5) at (-8.3,7.89) {$$};
\node[label={[xshift=0cm,yshift=0cm]16: $$}] (z6) at (3.05,-6.33) {$$};
\node[label={[xshift=0cm,yshift=0cm]17: $$}] (z7) at (-10.1,0) {$$};
\node[label={[xshift=0cm,yshift=0cm]18: $$}] (z8) at (-1,11.41) {$$};
\node[label={[xshift=0cm,yshift=0cm]19: $$}] (z9) at (-5.05,-6.33) {$$};

\node[label={[xshift=0cm,yshift=0cm]20: $$}] (w1) at (2.25,2.82) {$$};
\node[label={[xshift=0cm,yshift=0cm]21: $$}] (w2) at (0.45,5.08) {$$};
\node[label={[xshift=0cm,yshift=0cm]22: $$}] (w3) at (-2.45,5.08) {$$};
\node[label={[xshift=0cm,yshift=0cm]23: $$}] (w4) at (-4.25,2.82) {$$};
\node[label={[xshift=0cm,yshift=0cm]24: $$}] (w5) at (-3.6,0) {$$};
\node[label={[xshift=0cm,yshift=0cm]25: $$}] (w6) at (-1,-1.25) {$$};
\node[label={[xshift=0cm,yshift=0cm]26: $$}] (w7) at (1.6,0) {$$};

\foreach \source/\target in {z7/z4, z9/z3, z6/z8, z4/z5, z3/z7, z8/z9, z5/z6}
\draw[line width=1,shorten <=-2cm, shorten >=-2cm] (\source)--(\target);

\foreach \source/\target in {z7/w1, z9/w2, z6/w3, z4/w4, z3/w5, z8/w6, z5/w7}
\draw[line width=1,shorten <=-2cm, shorten >=-4.5cm] (\source)--(\target);

\end{tikzpicture}}
\vskip2.5cm

 \centering
 \captionof{figure}{The arrangement $\mathcal R(14)$ (projective picture)}
\label{picture3}
\end{minipage}

\bigskip

\begin{minipage}[t]{\linewidth\fboxsep\fboxrule}
\hspace*{3em}\raisebox{40em}

\centering
{\begin{tikzpicture}[scale=1]

\coordinate[label=150:$$] (a) at (8,0);
\coordinate[label=100:$$] (b) at (8,2);
\coordinate[label=50:$$] (c) at (8,-2);
\coordinate[label=40:$$] (g) at (8,3.78);
\coordinate[label=30:$$] (n) at (8,-5.21);
\coordinate[label=20:$$] (d) at (5.4,0.28);
\coordinate[label=20:$$] (e) at (5.4,2);
\coordinate[label=20:$$] (f) at (5.4,-2);
\coordinate[label=20:$$] (l) at (5.4,-2.99);
\coordinate[label=20:$$] (j) at (4.24,-2);
\coordinate[label=20:$$] (z) at (3.31,-2);
\coordinate[label=20:$$] (r) at (2.15,-2);
\coordinate[label=20:$$] (w) at (-0.45,-2);
\coordinate[label=20:$$] (t) at (2.15,2);
\coordinate[label=20:$$] (u) at (-0.45,5.21);
\coordinate[label=20:$$] (q) at (3.31,2);

\node[label={[xshift=0.1cm,yshift=-0.2cm]100:$p_1$}] (x1) at (8,-6.7) {$\bullet$};

\node[label={[xshift=-0.9cm,yshift=-0.18cm]35: $p_2$}] (x2) at (5.4,-6.7) {$\bullet$};
\node[label={[xshift=0.3cm,yshift=-0.2cm]500:$p_3$}] (x3) at (4.24,-6.7){$\bullet$};
\node[label={[xshift=0.2cm,yshift=0.8cm]600: $p_4$}] (x4) at (3.31,-6.7){$\bullet$};
\node[label={[xshift=-0.9cm,yshift=0.65cm]700:$p_5$}] (x5) at (2.15,-6.7){$\bullet$};
\node[label={[xshift=-0.7cm,yshift=0cm]800:$p_6$}] (x6) at (-0.45,-6.7){$\bullet$};

\draw[line width=2,shorten <=-2.5cm, shorten >=-2cm] (n)--(g) node[ pos=-0.31]{$\ell'_1$};
\draw[line width=2,shorten <=-2cm, shorten >=-1.5cm] (c) --(w)node[right, pos=-0.23]{$\ell_4$};
\draw[line width=2,shorten <=-2cm, shorten >=-7cm] (b) --(e) node[right, pos=-0.73]{$\ell_5$};
\draw[line width=2,shorten <=-4.6cm, shorten >=-3.7cm] (l) --(e)node[pos=-0.99]{$\ell'_2$};
\draw[line width=2,shorten <=-3cm, shorten >=-5cm] (b)--(d)--(r) node[right,pos=-1.55]{$\ell_6$};
\draw[line width=2,shorten <=-3cm, shorten >=-2cm] (c)--(d)--(q)--(u)node[right,pos=-1.85]{$\ell_3$};
\draw[line width=2,shorten <=-3cm, shorten >=-9cm] (n)--(l)node[right,pos=-0.9]{$\ell_2$};
\draw[line width=2,shorten <=-3.5cm, shorten >=-5cm] (n)--(t)node[right,pos=-0.37]{$\ell_1$};
\draw[line width=2,shorten <=-3cm,, shorten >=-2cm] (t)--(u);
\draw[line width=2,shorten <=-3cm,, shorten >=-9cm] (g)--(e)node[right,pos=-0.9]{$\ell_7$};
\draw[line width=2,shorten <=-5.6cm,, shorten >=-5.5cm](j)--+($(d)-(f)$) node[pos=-2.6]{$\ell'_3$};
\draw[line width=2,shorten <=-5.6cm,, shorten >=-5.5cm](z)--+($(d)-(f)$)node[pos=-2.6]{$\ell'_4$};
\draw[line width=2,shorten <=-5.6cm,, shorten >=-5.5cm](r)--+($(d)-(f)$)node[pos=-2.6]{$\ell'_5$};
\draw[line width=2,shorten <=-5.6cm,, shorten >=-6.5cm](w)--+($(d)-(f)$)node[pos=-2.6]{$\ell'_6$};
\draw[line width=1,shorten <=-2.9cm,, shorten >=-9.2cm](x1)--+($(z)-(j)$)node[pos=-4.8]{$x$};

 \end{tikzpicture}}
 
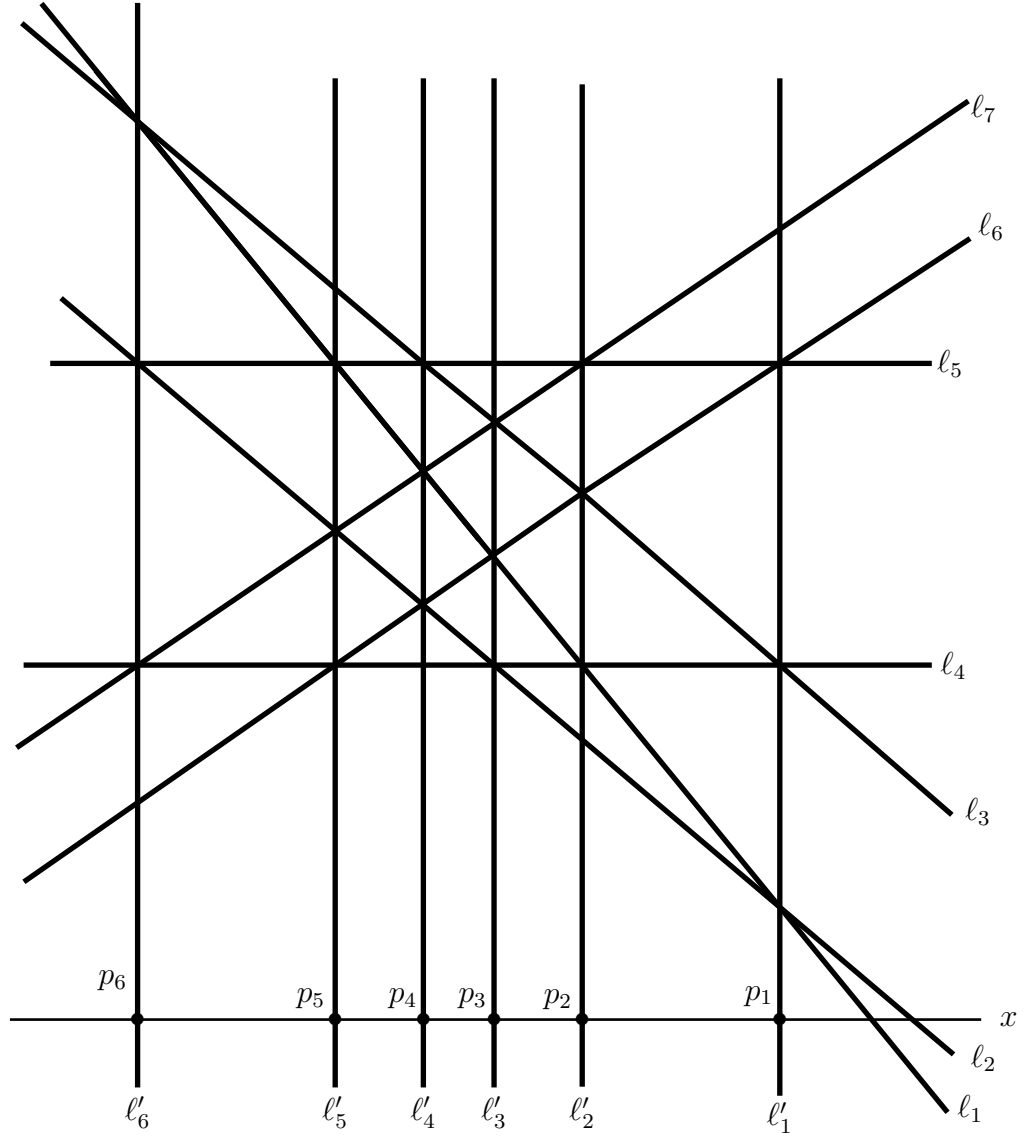
\captionof{figure}{The arrangement $\mathcal R(14)$ (affine picture) projected onto the $x$-axis}
\label{picture4}
\end{minipage}
 
\bigskip

\begin{lem}\label{permutations} 
\begin{enumerate}
\item For odd $n$ the local permutation $\sigma_i=\sigma(x_i,x'_i)$ is given by
$$ \sigma_i=(1,2)(3,4)\dots(n-2,n-1)(n)$$ 
for odd $i$ and by
$$\sigma_i=(1)(2,3)(4,5)\dots(n-1,n)$$
for even $i.$
\item For even $n$ the local permutation $\sigma_i=\sigma(x_i,x'_i)$ is given by
$$ \sigma_i=(1)(2,3)\dots(n-2,n-1)(n)$$ 
for odd $i$ and by
$$\sigma_i=(1,2)(3,4)\dots(n-1,n)$$
for even $i.$
\end{enumerate}
\end{lem}\qed

Let $\omega_n$ be a primitive $n$-th root of the unity.  We use results from the previous section to find explicit points in the characteristic variety.\bigskip

\begin{teo}\label{main1} The $n$ points $P_{n,k}\in (\C^*)^{2n-1},$ $k=0,\dots,n-1,$ whose coordinates are given by the $(s,t)$ with $s_j=\omega_n^k,\ j=1,\dots, n,$  $t_j=(\omega_n^k)^{n-2},\ j=1,\dots,n-1,$ belong to the characteristic variety.
\end{teo}
\begin{proof} The case $k=0$ is trivial since $P_{n,0}=(1,\dots,1)$ ($2n-1$ factors). So assume $k>0.$ Let us consider the row vector

\begin{equation}\label{vector} v_n=[(-1)^{n-1}\frac{\omega^{n-1}-1}{\omega-1},\dots,(-1)^j\frac{\omega^j-1}{\omega-1},\dots,-1]\in\C^{n-1}\end{equation}
where we set here for brevity $\omega=\omega_n^k.$

\noindent We prove that $v_n$ is an eigenvector for the transpose operator \ $^t\partial_1([\gamma_p])$ in (\ref{delta1}) for all $\gamma_p$ when the parameters take the above values. This shows that such parameters lower the rank  of $\partial_1$ in (\ref{conucleo}) so the corresponding point belongs to the characteristic variety.

Notice that by (\ref{permutations}) the local monodromy around $p_i$ coincides  with  the local monodromy $M_1$ around $p_1$  for odd $i$ or with the local monodromy $M_2$ around $p_2$ for even $i.$ Therefore by the expression (\ref{globalmonodromy}) for the monodromy it is sufficient  to prove that $v_n$ is eigenvector for $^t M_1$ and $^t M_2$ with parameters $s_i=\omega,$ with respect to the eigenvalue $\omega^{n-2},$  and both $\tau(x_1,x_2)^{-1}$  and $\tau(x_2,x_3)^{-1}$ take $v_n$ into an eigenvector of $^tM_2,\ ^tM_1$ respectively.

The first assertion is easily proved by looking at formulas (\ref{localmonodromy}). With the given parameters,  there are as many non-zero columns in $t_iM_i-Id$ as the number of triple points. So, for odd $n$ there are $(n-1)/2$ non-zero columns, in even position for $t_1M_1-Id,$ in odd position for $t_2M_2-Id;$ for even $n,$ there are  $(n-2)/2$ non-zero columns in even position for $t_1M_1-Id$ and $n/2$ non-zero columns in odd position for $t_2M_2-Id$ (see fig. \ref{picture2}, \ref{picture4}). The non-zero $j-$th column has entries $\omega^{n-2}-1$ at the $(j,j)$ entry,  preceded (if it is not the first entry of the column) by $\omega^{n-2}(1-\omega),$ and followed (if it is not the last entry of the column)  by $\omega^{n-1}-1;$ all other entries of the column vanish.  Then it is easy to  check that multiplying on the left by $v_n$  gives zero. 

The form of $T=\tau(x_1,x_2)^{-1}$ and $T'=\tau(x_2,x_3)^{-1}$ (with the given parameters) is derived from formula  (\ref{transport}) and from lemma \ref{permutations}.

In case $n$ is odd one obtains: for odd $j$  one has 
\bigskip

\centerline{$T_{jj}=-\omega^{-1}$ \ and  \ $T_{ij}=0$ \  for\  $ i\not=j;$}
\medskip

\noindent for even $j$  
\medskip

\centerline{$T_{j-1,j}=\omega^{-1},$ $T_{jj}=T_{j+1,j}=1,$ and $T_{ij}=0$ for all other $i$'s.} 
\bigskip

The shape of $T'$ is completely analogue by exchanging the role of even and odd $j.$

For $n$ even, the shape is the same exchanging  $T$ and $T'$.

Then one verifies directly that
$$v_n\cdot T\ =\ v_n\cdot T' \ =\ -\omega^{-1} \ v_n$$
and we conclude.
\end{proof}
\bigskip

\begin{rem} \label{action} By symmetry, there is an obvious action of the cyclic group $C_n= $ $<\sigma_n=(0,1,\dots,n-1)>$ onto the characteristic variety. Consider the edges of the $n-$gon  (in the projective picture) cyclically ordered as $\ell_0,\dots,\ell_{n-1}.$  Let us cyclically order also the diagonals $\ell'_0,\dots\ell'_{n-1},$ starting from the diagonal $\ell'_0$ which is orthogonal to $\ell_0.$ Take a new parameter $t_{n-1}$ for the line at infinity, with the condition that $\prod_{i=0}^{n-1} s_i\prod_{j=0}^{n-1} t_j =1.$  Then the action of $\sigma_n$ on some point $(s_0,\dots,s_{n-1},t_0,\dots,t_{n-1})\in(\C^*)^{2n}$ in $V$ is  the point $(s_{\sigma_n(0)},\dots,s_{\sigma_n(n-1)},t_{\sigma_n^2(0)},\dots,t_{\sigma_n^2(n-1)})$ which also belongs to $V.$ Notice that the points $P_{n,k}$ are fixed by this action.

With this numeration the double points are given by the intersection of the lines of indices $(i, (2i)')$ and the triple points are given by the intersections of the lines of indices $(i,j,(i+j)')$ (taking indices $mod\ n$).
\end{rem}
\bigskip

We need a lemma

\begin{lem}\label{zerodim}  Let $n\geq 5$ and let $\omega_n$ be a primitive $n-$th root of $1.$     
\begin{enumerate}
\item Let  $R_s=\C[s_0^{\pm 1},\dots,s_{n-1}^{\pm 1}]$ be the ring of Laurent polynomials in the variables $s_0,\dots s_{n-1}.$ The ideal
$$I\ =\ (s_is_j=s_l s_m:\ i+j\equiv l+m \ (mod\ n),\ i\not=j,\ l\not=m)$$
is one-dimensional with $n$ irreducible components
$$I_h\ =\ \{s_i=(\omega_n)^{hi} s_0: \ i=1,\dots, n-1\},\quad h=0,\dots,n-1.$$
\item Let $R_{s,t}=\C[s_0^{\pm 1},\dots,s_{n-1}^{\pm 1},t_0^{\pm 1},\dots,t_{n-1}^{\pm 1}]$
be the ring of Laurent polynomials in $s_i, t_j,$ $i,j=0,\dots,n-1.$  Let $J$ be the ideal 
$$J\ =\ (t_{i+j}s_i s_j=1\ ,\ \prod_{i=0}^{n-1}s_i\prod_{i=0}^{n-1} t_i=1)$$
where we take all possible pairs $(i,j)$ with $i\not=j,$ and the indices are mod $n.$ Then $J$ is $0$-dimensional 
and defines the $n^2$ points
$$P_{n,h,k}=\ \{s_i=(\omega_n)^{hi+k},  t_i=(s_0s_i)^{-1},\ i=0,\dots,n-1\},\ h, k=0,\dots,n-1.$$   
\end{enumerate}
\end{lem}

\begin{proof}
(1) First we prove the relations 
\begin{equation}\label{eq2} s_{i}^2=s_{j}s_{k}\quad  \forall\ i,j,k\ with\ 2i\equiv j+k\ (mod\ n).\end{equation}
By multiplying the two relations
$$s_0s_1=s_2s_{n-1},\ s_0s_{n-1}=s_1s_{n-2}\ (n>3)$$
we deduce $s_0^2=s_2s_{n-2}.$ Therefore  (\ref{eq2}) follows for $n>4.$ 
So we have  \ 
$s_is_j=s_ls_m\ \forall\  i,j,l,m\ with\ i+j\equiv\ l+m\  (mod\ n).$ \  
Then one easily deduces for recurrence the relations
$$s_{a_1}\dots s_{a_k}=s_{b_1}\dots s_{b_k}$$
every time $\sum a_j=\sum b_j \ (mod\ n).$ In particular one has 
$s_0^n=s_i^n$ for all $i=1,\dots,n-1.$ Then $s_1=(\omega_n)^h s_0$ for some $h.$ 
From $s_is_0=s_1s_{i-1}$ it follows by induction that $s_i=(\omega_n)^{hi}s_0$ which proves part (1).

For (2), notice that relations defining $I$ follow from those which define $J,$ so  we have again $s_i=(\omega_n)^{hi} s_0$, $i=0,\dots,n-1,$ for some $h$ in $0,\dots,n-1.$ 

Now using $t_i=(s_0s_i)^{-1}$ the last  relation gives
$\prod s_i \prod(s_0s_i)^{-1}=(s_0)^{-n}=1. $
Therefore $s_0=\omega_n^k$ for some $k$ in $0,\dots,n-1.$ 
 
\end{proof}

In conclusion, we show
\bigskip

\begin{teo}\label{isolato} The $\phi(n)$ points $P_{n,k},\ (n,k)=1,$ of theorem \ref{main1} are isolated points of the characteristic variety $V(\A).$
\end{teo}

\begin{proof} We divide the proof in several steps.

I. We already know from theorem \ref{main1} that $rk(\partial_1(P_{n,k}))<n-1, \ k=0,\dots, n-1.$  We directly check that  \ $rk(\partial_1(P_{n,k}))= n-2$ \ if \  $(n,k)=1.$ 

 In fact, the $(n-1)\times 2(n-1)$-submatrix of $\partial_1$ which corresponds to the local monodromy around the first two points $p_1,\ p_2$ is already of rank $n-2.$ This follows from the explicit description of the matrices $B_i:=t_iM_i-Id,\ i=1,2,$ given in theorem \ref{main1}. By taking a basepoint between the two points $p_1$ and $p_2$ (which amounts to simultaneously conjugate the blocks of the boundary matrix), the two blocks $B_1, \ B_2$ are submatrices of the corresponding boundary $\partial_1.$ Reordering the $n-1$ non-zero columns of $B_1$ and $B_2$ gives a tridiagonal matrix of order $n-1$ which has clearly rank $n-2$. 
 
 Therefore, $rk(\partial_1(P'))\geq n-2$  if  $P'$ is close to $P_{n,k}.$  We have to show that  $rk(\partial_1(P'))=n-1$ if $P'\not =P_{n,k}$ and $P'$ is close to $P_{n,k}.$   
\smallskip

II. It follows from the previous point that the left kernel of $\partial_1$ is of dimension $1$ in the given points, so it is spanned by the vector $v_n$ of theorem \ref{main1}. 

Notice that, under the hypothesis $(n,k)=1,$ all the coordinates of $v_n$ are different from $0.$ That implies that any $n-2$ rows of $\partial_1$ are linearly independent, and this must remain true in a neighborhood of $P_{n,k}.$
\smallskip

III. Recall from theorem \ref{charmono} that the vector $v_n$ is  a common  eigenvector for the transposed monodromy operators. Let us write the explicit form of the first block \  $t_1\ ^t\!M_1-Id$ \ of the boundary operator for $n=6,7$.

We have 
\medskip

$t_1\ ^t\!M_1-Id=$

$$= \begin{bmatrix} t_1-1 & t_1s_2(1-s_3) & 0 & 0 & 0 \\
0 & t_1s_2s_3-1 & 0 & 0 & 0 \\
0 & t_1(1-s_2) & t_1-1 & t_1s_4(1-s_5) & 0 \\
0 & 0 & 0 & t_1s_4s_5-1 & 0 \\
0 & 0 & 0 & t_1(1-s_4) & t_1-1 \end{bmatrix}   \quad \mbox{,for \ $n=6$;}$$

$$= \begin{bmatrix} t_1s_1s_2-1 & 0& 0 & 0 & 0 & 0 \\
t_1(1-s_1) & t_1-1&  t_1s_3(1-s_4) & 0 & 0 & 0 \\
0& 0 & t_1s_3s_4-1 & 0 & 0 & 0 \\
0 & 0 & t_1(1-s_3) & t_1-1 & t_1s_5(1-s_6) & 0 \\
0 & 0 & 0 & 0 & t_1s_5s_6-1 & 0 \\
0 & 0 & 0 & 0&  t_1(1-s_5) & t_1-1 \end{bmatrix}   \quad \mbox{,for \ $n=7$.}$$
\medskip

\noindent Having \ $(t_1\ ^t\!M_1-Id) v_n=0$ \ and all components of $v_n$ different from $0$ clearly gives  
$$t_1s_2s_3=t_1s_4s_5=1 \quad \mbox{for $n=6,$} $$
and   
$$t_1s_1s_2=t_1s_3s_4=t_1s_5s_6=1 \quad \mbox{for $n=7$}.$$ 
These conditions are equivalent to 
$$t_1s_is_j=1$$
each time $\ell'_1,\ \ell_i,\ \ell_j$  form a triple point.  Of course, we have the same conditions for all $n.$
\smallskip

IV.  By symmetry, or changing coordinates by taking as a basepoint one point $x$ close to the vertical line $\ell'_i$, we obtain similar conditions for any vertical line.

Remark also that we can take to infinity any other line passing through the center of the $n$-gon,
obtaining similar conditions for the line $\ell'_n$ with given parameter $t_n.$ In the projective situation, the product of all $s$ and $t$ equal $1.$
\smallskip

V.  It follows that $P_{n,k}$ is contained  in the zero locus of the ideal (with obvious notation)
\begin{equation}\label{ideal} I=(t_{\ell}s_{\ell'}s_{\ell''}=1\ :\  \mbox{each time}\ \ell\cap\ell'\cap\ell''\ \mbox{is a triple point of}\ \arr)\cap (\prod\ s_i\prod t_j=1)\end{equation}
Now notice that any point $P'$ very close to $P_{n,k}$ either does not belong to the characteristic variety or it corresponds to a common eigenvector $v'_n$ still having non vanishing components, so defining the same conditions.  

It follows that in a neighborhood of $P_{n,k}$ the characteristic variety is contained in the zero locus  of (\ref{ideal}).       

By lemma  \ref{zerodim} the ideal  $I$ is zero-dimensional (for $n\geq 5$) so the theorem is proved.
%
\end{proof}
\bigskip

\begin{rmk} \label{nonprimo}
\begin{enumerate}

\item In \cite{toryos}, pages 45--46, the case $n=8$ is considered and the authors claim that the six points $P_{8,k},\ k=1,\dots,7,\ k\not=4,$ are isolated points in the characteristic variety. The argument used in theorem \ref{isolato} works for $(n,k)=1;$ nevertheless, it can be refined in the following way. For $n=8,$ $k=2, 6,$ the eigenvector $v_8$ in (\ref{vector}) has a unique zero component in fourth position. An argument similar to that used in the proof of theorem \ref{isolato} produces a list of equalities among eigenvalues of the monodromy operators, giving rise to an ideal $I'$ smaller than 
the ideal $I$ given in (\ref{ideal}). Making explicit computations, we find that $I'$ is $1$-dimensional. Finally, we find a $1$-dimensional translated component of the characteristic variety which contains the two points $P_{8,2},\ P_{8,6}$ which are actually not isolated.
If we order the lines as in remark  \ref{action}, such $1$-dimensional translated component is given as follows
\medskip

\begin{center}
\begin{equation}\label{translated8}
\begin{array}{lclc}
s_i=x&  \mbox{for even} & i & \\
s_i=-x^{-1}&  \mbox{for odd} & i & \\
t_i=-1 & \mbox{for odd} & i  & \\
t_0=t_4=x^2 & & &\\
t_2=t_6=x^{-2} & & & \\
\end{array}
\end{equation}
\end{center}
$( x\in\C^* )$.

\medskip

Case $n=4,$ which corresponds to  the $B_3$-deleted arrangement considered in \cite{suciu}, can also be included in our method. If we consider any of the points  $P_{n,k},$ $k=1,2,3,$ we obtain the same ideal $I$ in (\ref{ideal}), which is in this case $1$-dimensional.   
 In fact,  the zero locus of $I$ has four irreducible $1$-dimensional components: one of these has the same shape as (\ref{translated8}), being parametrized as 
\begin{equation}\label{old} {}\{(x,-x^{-1},x,-x^{-1},x^2,-1,x^{-2},-1) {}:{} x\in\C^* \} \end{equation}
and it is the translated component appearing in \cite{suciu}.  It contains the two points $P_{4,1}, \ P_{4,3}.$  The other three components  of the zero locus of $I$ do not belong to the characteristic variety. 

\item Some cases where $(n,k)\not=1$ are known.
 For example, in case $n=9, \ k=3,$ the point $P_{9,3}$  is contained into a $2$-dimensional component coming from a multinet (see \cite{falkyuz}). 
\end{enumerate}
\end{rmk}

The argument outlined in rmk \ref{nonprimo} can be extended to any case $n=4m$ to produce a $1$-dimensional translated component, generalizing (\ref{old}), (\ref{translated8}). 

\begin{teo} \label{onedimensional} When $n=4m$ the characteristic variety $V(\A)$ contains a $1$-dimensional translated component parametrized as (the labelings are as in remark \ref{action})
\begin{equation}\label{general1}
\begin{array}{ll}
s_{2i}=x, \  s_{2i+1}=-x^{-1}, & i=0,\dots,2m-1;  \\
t_i=-1,\  \mbox{$i$ odd;}&  \\
 t_{i}=x^2,\ i\equiv0\ (mod \ 4); & \ t_{i}=x^{-2},\ i\equiv 2\ (mod\ 4)  
 \end{array}
 \end{equation}
\end{teo}
 
\begin{proof} We outline the proof of the theorem. 
First, one shows, by using the explicit form of the boundary as in theorem \ref{main1}, that the points in \ref{general1} lower the rank of the boundary matrix, so the component is contained in $V(\A).$ 

To prove that this is an isolated component, we proceed in a way similar to the proof of theorem~\ref{isolato}.  

Let us consider the point $P_{n,m}$ (or $P_{n,3m})$ which is contained in (\ref{general1}).  The vector $v_n$ of (\ref{vector}) has zero entries exactly in the positions multiple of $4$. The argument of theorem \ref{isolato} produces an ideal $I'$ which is obtained by $I$ by deleting, for each $i$ even in $0,\dots,4m-1, $ the $m-1$ equations  \ $t_{i}s_{\frac{i}{2}+2j}s_{\frac{i}{2}-2j}=1,$\ $j=1\dots,m-1$ (indices are taken $mod\ n$). One sees that the zero locus $Z(I')$ is of dimension $1$   and  (\ref{general1}) is contained in $Z(I')$ as an irreducible component which contains $P_{n,m}.$

Then we conclude that such component is not contained in a higher dimensional one by the same reasoning as that in theorem \ref{isolato}, using that points of the characteristic variety which are close to $P_{n,m}$ impose the same conditions.   

\end{proof}
\bigskip

\begin{rmk}
Analogue computations can be done for other points $P_{n,k},$ where the number and the  positions of the zero entries of the common eigenvector tell us what ideal to take. 
For example, we made some computations finding further isolated points in $V(\A)$ for $n=12.$  Details and  further applications will be provided  later.  

\end{rmk}
      
\vskip2cm

%
%

%
%
%

\end{document}